\documentclass[11pt]{amsart}
\usepackage[foot]{amsaddr}
\usepackage[margin=1in,letterpaper,portrait]{geometry}
\usepackage{ytableau,shuffle}
\usepackage[enableskew,vcentermath]{youngtab}
\usepackage{latexsym,amsmath,amssymb,verbatim,tikz}
\usepackage[pdfencoding=auto, psdextra]{hyperref}
\usepackage{graphicx,xcolor,accents}

\usepackage{randtext}

\makeatletter
\DeclareRobustCommand*{\bfseries}{%
	\not@math@alphabet\bfseries\mathbf
	\fontseries\bfdefault\selectfont
	\boldmath
}
\makeatother

\usepackage[english]{babel}
\usepackage{amssymb}

\usepackage{amsmath}
\usepackage{amsthm}
\usepackage{graphicx}
\usepackage[style=ieee, sorting=nyt,citestyle=numeric]{biblatex}
\usepackage{csquotes}
\usepackage{shuffle}
\usepackage{bm}
\usepackage{bbm}

\usepackage{dirtytalk}

\newtheorem{theorem}{Theorem}[section]
\newtheorem{corollary}[theorem]{Corollary}
\newtheorem{lemma}[theorem]{Lemma}

\newtheorem{conjecture}[theorem]{Conjecture}
\newtheorem{observation}[theorem]{Observation}
\newtheorem{question}[theorem]{Question}

\theoremstyle{definition}
\newtheorem{definition}[theorem]{Definition}

\theoremstyle{remark}
\newtheorem*{remark}{Remark}

\newtheorem*{theorem*}{Theorem}
\newtheorem*{claim*}{Claim}

\newcommand{\sym}{\mathrm{Sym}}
\newcommand{\qsym}{\mathrm{QSym}}
\newcommand{\Des}{\operatorname{Des}}
\newcommand{\rank}{\operatorname{rank}}
\newcommand{\ZZ}{\mathbb{Z}}
\newcommand{\NN}{\mathbb{N}}
\newcommand{\RR}{\mathbb{R}}
\newcommand{\QQ}{\mathbb{Q}}
\newcommand{\F}{\mathcal{F}}
\newcommand{\B}{\mathcal{B}}

\makeatletter
\newcommand{\proofstep}[1]{%
  \par
  \addvspace{\medskipamount}
  \textbf{#1\@addpunct{.}}\enspace\ignorespaces
}
\makeatother

\begin{document}

\title{Tight Lower Bound for Pattern Avoidance and Symmetric Functions}
\author{Avichai Marmor}
\address{Department of Mathematics, Bar-Ilan University, Ramat-Gan 52900, Israel}
\email{avichai@elmar.co.il}
\date{July 22, 2023}
\thanks{Partially supported by the Israel Science Foundation, Grant No.\ 1970/18, and by the European Research Council under the ERC starting grant agreement No.\ 757731 (LightCrypt).}

\begin{abstract}
For a set of permutations (patterns) $\Pi$ in $S_k$, consider the set of permutations in $S_n$ that avoid all patterns in $\Pi$. In current algebraic combinatorics, a significant problem is to identify pattern sets $\Pi$ for which the corresponding quasisymmetric function is symmetric for all $n$. Recently, Bloom and Sagan proved that unless $\Pi \subseteq \{12\dots k, k \dots 21\}$, the size of such $\Pi$ must be at least $3$ for any $k \ge 4$. They also posed a general lower bound conjecture.


In this work, we resolve this conjecture and give a tight lower bound, namely, the minimal size of such $\Pi$ is exactly $k - 1$. The proof relies on a novel generalization of Bose's theorem in extremal combinatorics, utilizing the multilinear polynomial approach introduced by Alon, Babai, and Suzuki.
\end{abstract}


\maketitle

\tableofcontents

\section{Introduction}
Let $S_n$ denote the group of permutations on the set $[n] := \{1, \dots, n\}$. We represent permutations using one-line notation, and denote $\sigma_i := \sigma(i)$. For example, if $\sigma = 51243 \in S_5$ then $\sigma_1 = 5$ and $\sigma_2 = 1$. We also define two special permutations in $S_n$, namely the increasing permutation $\iota_n = 12 \dots n = \mathrm{Id}_n$ and the decreasing permutation $\delta_n = n \dots 21$. These permutations are referred to as the \emph{monotone elements}.

Let $k \le n$ and let $\pi \in S_k$ and $\sigma \in S_n$. We say that $\sigma$ \emph{contains the pattern} $\pi$, if there exist indices $1 \le i_1 < \dots < i_k \le n$ such that for all $1 \le j < j' \le k$,
\[ \pi_{j} < \pi_{j'} \Longleftrightarrow \sigma_{i_j} < \sigma_{i_{j'}}. \]
Otherwise (or if $k > n$), we say that $\sigma$ \emph{avoids} $\pi$. For instance, consider the permutation $\sigma = 7 \bm{2} \bm{6} 8 \bm{1} 5 \bm{4} 3 \in S_8$. It contains the pattern $2 4 1 3 \in S_4$ due to the subsequence $2 6 1 4$. The set of permutations in $S_n$ that avoid every $\pi \in \Pi$ is denoted $S_n (\Pi)$.

The study of pattern avoidance in permutations has a rich history. A notable early result in this area is the Erdős and Szekeres' theorem~\cite{erdos_szekeres}, which states (as a special case) that every $\sigma \in S_{a^2+1}$ contains a monotone subsequence of length $a+1$. In other words, for all $n \geq a^2+1$, we have $S_n (\{\iota_{a+1}, \delta_{a+1} \}) = \emptyset$. Enumerative problems related to pattern avoidance have also received significant attention. For instance, MacMahon~\cite{avoid123} showed that $|S_n(\{123\})|=|S_n(\{321\})|$ is equal to the Catalan number, and Knuth~\cite{avoid132} extended this result to $|S_n(\{\pi\})|$ for any $\pi \in S_3$.

A set $\B \subseteq S_n$ is referred to as \emph{symmetric} or \emph{Schur-positive} if its quasisymmetric generating function, as defined in Section~\ref{sec:preliminaries}, exhibits symmetry or Schur-positivity, respectively.
The study of symmetric functions is partially motivated by the correspondence between symmetric functions of degree $n$ and class functions on $S_n$, which is established through the Frobenius characteristic map described in~\cite[Chapter 4.7]{sagan_book}. A symmetric function is Schur-positive if its corresponding class function is a proper character (see~\cite{adin_roichman_fine_set} for a detailed explanation). This perspective provides a rich algebraic framework for studying symmetric and Schur-positive sets. The investigation of such sets was initiated in the seminal works of Gessel~\cite{knuth_classes_schur_positive} on Knuth-classes and Gessel---Reutenauer~\cite{gessel1993counting} on conjugacy classes.





In an effort to establish a connection between pattern avoidance and algebraic structures associated with Schur-positivity, the search for symmetric pattern-avoiding sets of permutations has gained interest. One notable example of such a Schur-positive set is the set of arc permutations, introduced by Elizalde and Roichman~\cite{arc_pattern_avoiding}. This set consists of permutations $\sigma \in S_n$ that avoid all patterns $\pi \in S_4$ with $|\pi(1) - \pi(2)| = 2$.

To further explore the concept, we introduce the notion of \emph{symmetrically avoided} sets:

\begin{definition}
Given $\Pi \subseteq S_k$, we say that $\Pi$ is \emph{symmetrically avoided} if the pattern-avoiding set $S_n(\Pi)$ is symmetric for all $n$.
\end{definition}

Elizalde and Roichman's result on arc permutations mentioned above can be equivalently stated by noting that the set of permutations $\{\pi \in S_4 \mid |\pi(1) - \pi(2)| = 2\}$ is symmetrically avoided. Building upon this result, Sagan and Woo posed the following problem:

\begin{question}[Sagan and Woo~\cite{problem_find_pattern}]
Which subsets $\Pi \subseteq S_k$ are symmetrically avoided?
\end{question}

Sagan and coauthors~\cite{problem_find_pattern, S3_patterns_list, BloomSagan} initiated the study of \emph{small} symmetrically avoided sets.
Using the Robinson---Schensted correspondence, it can be shown that the singleton sets $\Pi = \{\iota_k\}$ and $\Pi = \{\delta_k\}$, which consist of the increasing and decreasing elements, respectively, are symmetrically avoided (see~\cite[Section 2]{S3_patterns_list} for more details). Furthermore, it can be shown that $S_n(\Pi)$ is Schur-positive for these singleton sets. Bloom and Sagan~\cite[Theorem 2.3]{BloomSagan} proved that these are the only symmetrically avoided singleton sets.

By combining the results of Hamaker---Pawlowski---Sagan \cite[Theorem 1.2]{S3_patterns_list} and Bloom---Sagan \cite[Theorem 2.6]{BloomSagan}, a complete classification of $2$-element symmetrically avoided sets can be obtained.
In particular, it turns out that for every $k \ge 4$ and every $\Pi \subseteq S_k$ with two elements, $\Pi$ is symmetrically avoided if and only if $\Pi = \{\iota_k, \delta_k\}$.

Based on these findings, Bloom and Sagan formulated the following conjecture:

\begin{conjecture}[{\cite[Conjecture 2.7]{BloomSagan}}]
\label{conj:main conjecture}
For each $p \ge 3$, there exists a $K = K(p)$, such that if $|\Pi| = p$ and $\Pi \subseteq S_k$ for some $k \ge K$, then $\Pi$ is not symmetrically avoided.
\end{conjecture}

We prove that the conjecture holds with $K(p) = p+2$:
\begin{theorem}
\label{thm:my pattern theorem}
Let $p \ge 3$ and $k \ge p+2$ be positive integers, and let $\Pi \subseteq S_k$ with $|\Pi| = p$. Then $\Pi$ is not symmetrically avoided.
\end{theorem}

The bound is tight: For every $p \ge 3$, there exists a symmetrically avoided set $\Pi \subseteq S_{p+1}$ with $|\Pi| = p$. 
Such a set was discovered by Hamaker, Pawlowski and Sagan \cite[Section 5]{S3_patterns_list} and is discussed in Corollary~\ref{cor:minimal avoidance example}.

The proof utilizes a new generalization of a theorem by Bose~\cite{ray_chaudhuri_wilson_single_intersection_size} in extremal combinatorics, which bounds the size of an intersecting family. Before we proceed, let us define an intersecting family:

\begin{definition}
\label{def:intersecting family}
    Let $\F = \{A_1, \dots, A_m \}$ be a family of $m$ sets.
    \begin{enumerate}
        \item We say that $\F$ is \emph{$k$-uniform} if for all $i$, $|A_i| = k$.
        \item We say that $\F$ is \emph{$\ell$-intersecting} if for all $i \ne j$, $| A_i \cap A_j| = \ell$.
        \item We say that $\F$ is \emph{$(\ell_1, \ell_2)$-intersecting} if the following holds for all $i \ne j$:
            \begin{itemize}
                \item If $|i - j| = 1$, then $|A_i \cap A_j| = \ell_1$.
                \item If $|i - j| \ge 2$, then $|A_i \cap A_j| = \ell_2$.
            \end{itemize}
    \end{enumerate}
\end{definition}


In 1949, Bose presented an influential theorem that addressed the size of intersecting families:

\begin{theorem}[Bose~\cite{ray_chaudhuri_wilson_single_intersection_size}]
\label{thm:his extremal theorem}
Let $\F = \{A_1, \dots, A_m \}$ be a $k$-uniform family consisting of $m$ distinct subsets of $[n]$. If $\F$ is $\ell$-intersecting for some $\ell \ge 0$, then $m \le n$.
\end{theorem}


Ray-Chaudhuri and Wilson later generalized this theorem in 1975~\cite{extremal_theorem}. Another proof of the Ray-Chaudhuri---Wilson theorem was presented by Alon, Babai, and Suzuki in 1991~\cite[Section 2]{alon_proof}, where they employed multilinear polynomials and analyzed their linear structure.


In this paper, we extend the scope of Bose's theorem in a different direction:

\begin{theorem}
\label{thm:my extremal theorem}
Let $\F = \{A_1, \dots, A_m\}$ be a $k$-uniform family consisting of $m$ distinct subsets of $[n]$. If $\F$ is $(\ell_1, \ell_2)$-intersecting for some $\ell_1, \ell_2 \ge 0$, then $m \le n$.
\end{theorem}


We present two proofs of Theorem~\ref{thm:my extremal theorem}. The first proof applies the multilinear polynomial method introduced by Alon, Babai and Suzuki, and the second applies matrix theory.

The remainder of this paper is organized as follows: Section~\ref{sec:preliminaries} provides the necessary background information, while Section~\ref{sec:proof extremal theorem} presents the proofs of Theorem~\ref{thm:my extremal theorem}. In Section~\ref{sec:proof pattern theorem}, we apply our generalized theorem to address the problem of symmetric pattern-avoiding sets. Finally, in Section~\ref{sec:conclusion} we conclude with further research.

\section{Preliminaries}
\label{sec:preliminaries}

\begin{definition}
Given $n \in \NN$, a \emph{partition} $\lambda$ of $n$ (denoted $\lambda \vdash n$), is a weakly-decreasing sequence $\lambda = (\lambda_1 \ge \lambda_2 \ge \dots \ge \lambda_k)$ of positive integers, such that $\lambda_1 + \dots + \lambda_k = n$.

A \emph{composition} $\alpha$ of $n$ (denoted $\alpha \vDash n$) is a sequence $\alpha = (\alpha_1, \alpha_2, \dots, \alpha_k )$ of positive integers, such that $\alpha_1 + \dots + \alpha_k = n$. Every partition is also a composition.
\end{definition}

Compositions of $n$ are in bijection with subsets of $[n-1] := \{1, \dots, n-1\}$ through the mapping:
\[
    [n-1] \supseteq \{i_1, i_2, \dots, i_k \} \mapsto (i_1, i_2 - i_1, \dots, i_k - i_{k-1}, n - i_k) \vDash n.
\]
For example, for $n = 10$, $\{2, 4, 7, 8\} \mapsto (2, 2, 3, 1, 2)$. The composition associated with a set $S$ is denoted $\alpha_S$, and the set associated with a composition $\alpha$ is denoted $S_\alpha$.

For two compositions $\alpha, \beta \vDash n$, we say that:
\begin{itemize}
    \item $\beta$ is a \emph{refinement} of $\alpha$ (denoted $\beta \le \alpha$) if $S_{\alpha} \subseteq S_{\beta}$.
    \item $\alpha$ and $\beta$ are \emph{equivalent} (denoted $\alpha \sim \beta$), if $\beta$ is a rearrangement of the elements of $\alpha$.
\end{itemize}

Denote the ring of symmetric functions by $\sym$ and the ring of quasisymmetric functions by $\qsym$ (see, for example, \cite[Section 1]{S3_patterns_list} for details). The space of homogeneous symmetric functions of degree $n$ is denoted by $\sym_n$, and the space of homogeneous quasisymmetric functions of degree $n$ is denoted by $\qsym_n$.

The space $\sym_n$ has a basis consisting of the \emph{monomial symmetric functions}:
\[ m_\lambda := \sum_{i_1,i_2, \dots, i_k \text{ distinct}} x_{i_1}^{\lambda_1} \cdots x_{i_k}^{\lambda_k}, \]
where $\lambda \vdash n$ is a partition. For example:
\[
    m_{3,1} = x_1^3 x_2 + x_2^3 x_1 + x_1^3 x_3 + x_3^3 x_1 + x_2^3 x_3 + x_3^3 x_2 + \dotsb
\]

The space $\sym_n$ also has another important basis consisting of the Schur functions $\{s_\lambda \mid \lambda \vdash n\}$. The definition of Schur functions can be found in \cite[Section 4.4]{sagan_book}.

The space $\qsym_n$ has a basis consisting of the \emph{monomial quasisymmetric functions}:
\[
    M_\alpha := \sum_{i_1<i_2< \dots < i_k} x_{i_1}^{\alpha_1} \cdots x_{i_k}^{\alpha_k},
\]
where $\alpha \vDash n$ is a composition. For example:
\[
    M_{1,3} = x_1 x_2^3  + x_1 x_3^3 + x_2 x_3^3 + \dotsb,\quad \text{and} \quad M_{3,1} = x_1^3 x_2  + x_1^3 x_3 + x_2^3 x_3 + \dotsb
\]

\begin{observation}
\label{obs:decomposition monomial bases}
For every $\lambda \vdash n$:
\[ m_{\lambda} = \sum_{\alpha \sim \lambda} M_\alpha \]
\end{observation}

Observation~\ref{obs:decomposition monomial bases} has a simple implication, which we believe is likely known but may not be explicitly stated in the literature.

\begin{corollary}
\label{cor:symmetric function criterion}
    The quasisymmetric function $f=\sum_\alpha c_\alpha M_\alpha$ is symmetric if and only if
\[ \forall \alpha, \beta \vDash n: \alpha \sim \beta \implies c_\alpha = c_\beta. \]
\end{corollary}

The space $\qsym_n$ has an additional important basis consisting of the \emph{fundamental quasisymmetric functions}
\[ F_S := \sum_{\begin{matrix} i_1 \le i_2 \le \dots \le i_n \\ \forall j \in S:i_j < i_{j+1} \end{matrix}} x_{i_1} x_{i_2} \cdots x_{i_n},\quad S \subseteq [n - 1].\]
We also denote $F_{\alpha} = F_{S_\alpha}$ for a composition $\alpha \vDash n$.

\begin{observation}
\label{obs:decomposition fundamental and monomial}
\[
F_{\alpha} = \sum_{\beta \le \alpha} M_\beta,
\]
or equivalently:
\[
F_{S} = \sum_{\beta : S_\beta \supseteq S} M_\beta.
\]
\end{observation}


\begin{definition}
The \emph{descent set} of a permutation $\pi \in S_n$ is defined as
\[ \Des(\pi) := \{i \in [n-1]: \pi_i > \pi_{i+1}\}. \]
\end{definition}

\begin{definition}
\label{def:generating function}
Let $\B \subseteq S_n$ be a set of permutations. Its \emph{quasisymmetric generating function} is defined as
\[ \mathcal{Q}_n (\B) := \sum_{\pi \in \B} F_{\Des(\pi)}. \]
\end{definition}

\begin{definition}
\label{def:symmetric set}
Let $\B \subseteq S_n$. We say that $\B$ is \emph{symmetric} if $\mathcal{Q}_n (\B) \in \sym_n$.
If, in addition, the expansion of $\mathcal{Q}_n (\B)$ in the Schur basis has nonnegative coefficients, then $\B$ is called \emph{Schur-positive}.
\end{definition}

\section{Proof of Theorem~\ref{thm:my extremal theorem}}
\label{sec:proof extremal theorem}

We begin with a simple observation:
\begin{observation}
\label{obs:intersecting family inequality of parameters}
    Let $\F = \{A_1, \dots, A_m\}$ be a $k$-uniform, $(\ell_1, \ell_2)$-intersecting family consisting of $m$ distinct subsets of $[n]$, and assume that $m \ge 4$. Then $2 \ell_2 \le \ell_1 + k$. Equality holds if and only if $A_1 \cap A_2 \subseteq A_4 \subseteq A_1 \cup A_2$.
\end{observation}
\begin{proof}
    The statement follows directly from the following inequalities:
    \begin{eqnarray*}
        \ell_{1} & = & |A_{1}\cap A_{2}| \ge |A_{1}\cap A_{2}\cap A_{4}| = |A_4|-|(A_4\cap A_{1}^{c})\cup (A_4\cap A_{2}^{c})|\\
        & \ge & |A_4|-|A_{4}\cap A_{1}^{c}|-|A_{4}\cap A_{2}^{c}| = 2 \ell_{2}-k,
    \end{eqnarray*}
    where $A^c := [n] \setminus A$ denotes the complement set of $A$.
\end{proof}

As we will see, intersecting families with $2 \ell_2 = \ell_1 + k$ exhibit unique properties. Therefore, the general proof does not apply to these families. However, it turns out that such singular families are exceedingly rare:
\begin{lemma}
\label{lem:family with wierd intersections is small}
Let $\F$ be a $k$-uniform family containing $6$ sets, and let $\ell_1, \ell_2 \ge 0$ satisfy $2\ell_2 = \ell_1 + k$. Then $\F$ cannot be $(\ell_1, \ell_2)$-intersecting.
\end{lemma}
\begin{proof}
Assume, by contradiction, that $\F$ is $(\ell_1, \ell_2)$-intersecting. Observation~\ref{obs:intersecting family inequality of parameters} implies $A_1 \cap A_2 \subseteq A_4$. By applying a similar argument, it can be shown that for every $1 \le i \le 5$, and every $j$ that is neither adjacent to $i$ nor adjacent to $i+1$, $A_i \cap A_{i+1} \subseteq A_j$. Thus, $A_1 \cap A_2 \subseteq A_5$ and $A_1 \cap A_2 \subseteq A_6$, leading to $A_1 \cap A_2 \subseteq A_5 \cap A_6$. Combining this with the fact that $A_5 \cap A_6 \subseteq A_3$, we may deduce that $A_1 \cap A_2 \subseteq A_3$.

Therefore, $A_1 \cap A_2 \subseteq A_i$ for all $i$, resulting in $|\bigcap_i A_i | = \ell_1$. Removing the elements of $\bigcap_i A_i$ from all the sets yields a new intersecting family with $\ell_1 = 0$. Thus, we may assume, without loss of generality, that every two adjacent sets are disjoint, and $k = 2\ell_2$.

This case can be easily tested. Without loss of generality, let $A_1 = (0,2\ell_2] = \{1, \dots, 2\ell_2\}$. Since $A_1 \cap A_2 = \emptyset$, we assume, without loss of generality, that $A_2 = (2\ell_2,4\ell_2]$. With $A_2 \cap A_3 = \emptyset$ and $|A_1 \cap A_3| = \ell_2$, we may assume, without loss of generality, that $A_3 = (0,\ell_2] \cup (4\ell_2,5\ell_2]$. By similar reasoning, we find $A_4 = (\ell_2,3\ell_2]$ and $A_5 = (0,\ell_2] \cup (3\ell_2,4\ell_2]$. However, no set $A_6$ satisfies all the intersection properties, leading to a contradiction.
\end{proof}

To prove Theorem~\ref{thm:my extremal theorem} we also need the following lemma:
\begin{lemma}
\label{lem:M alpha invertible}
    Let $\alpha \in \QQ$ such that $-1 < \alpha < 1$. Consider the matrix $M_m(\alpha) \in \RR^{m\times m}$ defined as
    \[
        M_m(\alpha) := \begin{pmatrix}
        1 & \alpha & & 0 \\
        \alpha & 1 & \ddots & \\
         & \ddots & \ddots & \alpha \\
         0 & & \alpha & 1
        \end{pmatrix}.
    \]
    Then $M_m(\alpha)$ is invertible.
\end{lemma}

It is worth noting that the assumptions on $\alpha$ in Lemma~\ref{lem:M alpha invertible} are crucial for the invertibility of $M_m(\alpha)$. For instance, the matrices $M_2(\pm 1)$ and $M_3\left(\pm \frac{1}{\sqrt{2}}\right)$ are singular, demonstrating the necessity of the conditions $-1 < \alpha < 1$ and $\alpha \in \mathbb{Q}$.

\begin{proof}
The proof relies on analyzing the determinant $d_m(\alpha)$ of $M_m(\alpha)$. We first note that $d_m(\alpha)$ is a polynomial in $\alpha$ with integer coefficients. Specifically, the constant coefficient of $d_m(\alpha)$ is $d_m(0) = \det(I) = 1$, where $I$ is the identity matrix. Applying the Rational Root Theorem to $d_m(\alpha)$ under the assumption that $\alpha \in \QQ$, we conclude that the only possible solutions to $d_m(\alpha) = 0$ are $\alpha \in \frac{1}{\ZZ}$. Since $-1 < \alpha < 1$, it follows that $|\alpha| \leq \frac{1}{2}$.

For the remaining case, $|\alpha| \leq \frac{1}{2}$, we proceed with a proof inspired by the proof of Gershgorin's Circle Theorem~\cite{gershgorin}. Assume by contradiction that $M_m(\alpha)$ is singular, 
implying the existence of a non-zero vector $v \in \RR^m$ such that $M_m(\alpha) \cdot v = 0$.
We take the minimal index $1 \le i \le m$ such that $|v_i| = \max\{|v_j| \mid 1 \le j \le m\}$, and assume, without loss of generality, that $v_i = 1$. It follows that $|v_j| < 1$ for all $j < i$ and $|v_j| \le 1$ for all $j > i$.

If $v_1 = 1$, then we have $0 = (M_m(\alpha) \cdot v)_1 = v_1 +\alpha v_2$. This implies $1 = |-v_1| = |\alpha | \cdot |v_2| \le \frac{1}{2}$, which leads to a contradiction. Thus, $v_1 < 1$. The case where $v_m = 1$ is handled similarly.

Assuming that $1<i<m$, we have $0 = (M_m(\alpha) \cdot v)_i = v_i + \alpha(v_{i-1}+v_{i+1})$. Recall that $\alpha \ne 0$ because $\alpha \in \frac{1}{\ZZ}$. Therefore, we can rewrite the equation as $-\frac{1}{\alpha}=v_{i-1}+v_{i+1}$. This implies that $2 \le \frac{1}{|\alpha|} \le |v_{i-1}|+|v_{i+1}| \le 2$. Thus, we can conclude that all the inequalities hold as equalities. In particular, we have $|v_{i-1}| = 1$, which contradicts the minimality of $i$. Hence, we can deduce that $M_m(\alpha)$ is invertible, and the proof is complete.
\end{proof}

Now we are ready to prove Theorem~\ref{thm:my extremal theorem}. We will present two proofs: the first applies the multilinear polynomial method, while the second employs matrix theory.
\begin{proof}[Proof of Theorem~\ref{thm:my extremal theorem}]
First, note that the sets are distinct, implying $0 \le \ell_1, \ell_2 < k$.

The theorem obviously holds for $m \le 3$. Hence, we focus on the case where $m \ge 4$. By Observation~\ref{obs:intersecting family inequality of parameters}, we have $2 \ell_2 \le \ell_1 + k$.

Let us begin with the case where $2 \ell_2 = \ell_1 + k$. By Lemma~\ref{lem:family with wierd intersections is small}, we obtain that $m \le 5$. Assume, by contradiction, that $m > n$, leading to $n \le 4$. Since $n < m \le \binom{n}{k}$, we deduce that $n = 4$ and $k = 2$. As $2\ell_2 = \ell_1 + k = \ell_1 + 2$ and $\ell_2 < k = 2$, we conclude that $\ell_1 = 0$ and $\ell_2 = 1$. Without loss of generality, we may assume that $A_1 = \{1,2\}$ and $A_2 = \{3,4\}$.
However, it can be easily verified that no set $A_3 \subseteq [4]$ satisfies both $A_2 \cap A_3 = \emptyset$ and $|A_{1}\cap A_{3}| = 1$, contradicting the assumption that $\ell_1 = 0$ and $\ell_2 = 1$.

From now on we will assume that $2 \ell_2 < \ell_1 + k$. For each $1 \le i \le m$, we define a multivariate polynomial $f_i:\RR^n \to \RR$ as follows:
\[ f_i (\bm x) = \sum_{j \in A_i} x_j - \ell_2. \]
Additionally, we define the polynomial
\[ g (\bm x) = \sum_{j=1}^n x_j - k. \]

The set $\{f_1, \dots, f_m, g\}$ consists of $m+1$ elements, all belonging to the linear space of polynomials of degree at most $1$. The dimension of this space is $n+1$. Thus, to establish the desired result that $m \leq n$, it suffices to prove the following claim.

\begin{claim*}
The set $\{f_1, \dots, f_m, g\}$ is linearly independent over $\RR$.
\end{claim*}

In order to prove the claim, we associate indicator vectors with sets using the notation $(\bm{v_i})_j = \mathbbm{1}_{j \in A_i}$. Importantly, for every $i_1,i_2 \in [m]$, we observe that:
\[ f_{i_1}(\bm{v_{i_2}}) = \begin{cases}
			k - \ell_2, & i_1 = i_2\\
            \ell_1 - \ell_2, & |i_1 - i_2 | = 1\\
            0, & |i_1 - i_2 | \ge 2
		 \end{cases}\ . \]

Furthermore, we have $g(\bm{v_i}) = 0$ for all $i$. Define an arbitrary $\bm{w} \in \RR^n$ such that $g(\bm w) \ne 0$ (e.g., $\bm w = 0$).

Let us define a linear transformation $T$ from the polynomial space to $\RR^{m+1}$, by
\[ T:f \mapsto
    \begin{pmatrix}
       f(\bm{v_1}) \\
       \vdots \\
       f(\bm{v_m}) \\
       f(\bm w)
    \end{pmatrix}.
\]

To establish the claim, it remains to show that the set $\{T(f_1), \dots, T(f_m), T(g)\}$ forms a basis of $\RR^{m+1}$. The representing matrix of this set is:
\[ \begin{pmatrix}
k-\ell_2 & \ell_1-\ell_2 & & 0 & 0 \\
\ell_1-\ell_2 & k-\ell_2 & \ddots & & 0 \\
 & \ddots & \ddots & \ell_1-\ell_2 & \vdots \\
 0 & & \ell_1-\ell_2 & k-\ell_2 & 0 \\
 * & * & \dots & * & \ne 0
\end{pmatrix}.
\]

It suffices to show that this matrix is invertible, indicating the desired result. The matrix is invertible if and only if the minor obtained by deleting the last row and the last column is invertible, which can be represented as:
\[ A := \begin{pmatrix}
k-\ell_2 & \ell_1-\ell_2 & & 0 \\
\ell_1-\ell_2 & k-\ell_2 & \ddots & \\
 & \ddots & \ddots & \ell_1-\ell_2 \\
 0 & & \ell_1-\ell_2 & k-\ell_2
\end{pmatrix}.
\]
Notably, $A = (k-\ell_2)M_m(\alpha)$, where $\alpha = \frac{\ell_1-\ell_2}{k-\ell_2} \in \QQ$ and $M_m(\alpha)$ is the matrix defined in Lemma~\ref{lem:M alpha invertible}. It is worth noting that $-1 < \alpha < 1$ since $\ell_1 < k$ and $2\ell_2 < \ell_1+k$. Consequently, we may apply Lemma~\ref{lem:M alpha invertible} and conclude that $M_m(\alpha)$ is invertible. Thus, $A = (k-\ell_2)M_m(\alpha)$ is also invertible, thereby completing the proof.
\end{proof}

\begin{proof}[Alternative proof of Theorem~\ref{thm:my extremal theorem}]
Note that the sets $A_i$ are distinct, which implies $0 \leq \ell_1, \ell_2 < k$.

For $m \leq 3$, the theorem is clearly true. Hence, we focus on the case where $m \geq 4$. According to Observation~\ref{obs:intersecting family inequality of parameters}, we have $2\ell_2 \leq \ell_1 + k$.

Let us consider the case where $2\ell_2 = \ell_1 + k$. By Lemma~\ref{lem:family with wierd intersections is small}, we conclude that $m \leq 5$. 
Therefore, any potential counterexample for the statement must satisfy $n \leq 4$. Upon examining all families with $n \leq 4$, we can verify that none of them contradicts the statement.

Henceforth, we assume that $2\ell_2 < \ell_1 + k$. Recall that we are given $m$ sets $A_i \subseteq [n]$, for $ 1 \le i \le m$. We define a matrix $A \in \RR^{m \times n}$ such that $A_{ij} = \mathbbm{1}_{j \in A_i}$. Notably, the matrix $AA^t$ has entries \[(AA^t)_{i_1 i_2} = |A_{i_1} \cap A_{i_2}| = \begin{cases}
			k, & i_1 = i_2\\
            \ell_1, & |i_1 - i_2 | = 1\\
            \ell_2, & |i_1 - i_2 | \ge 2
		 \end{cases}\ . \]
   
Moreover, for any $n_1, n_2 \in \NN$, let $J_{n_1,n_2} \in \RR^{n_1 \times n_2}$ denote the all-ones matrix. Since $|A_i| = k$ for all $i$, we have $A \cdot J_{n,m} = k\cdot J_{m,m}$. Therefore, we obtain
\[
    B := A(A^t-\frac{\ell_2}{k}J_{n,m}) = \begin{pmatrix}
k-\ell_2 & \ell_1-\ell_2 & & 0 \\
\ell_1-\ell_2 & k-\ell_2 & \ddots & \\
 & \ddots & \ddots & \ell_1-\ell_2 \\
 0 & & \ell_1-\ell_2 & k-\ell_2
\end{pmatrix},
\]
where $B \in \RR^{m \times m}$. Importantly, $B = (k-\ell_2)M_m(\alpha)$, where $\alpha = \frac{\ell_1-\ell_2}{k-\ell_2} \in \QQ$, $-1 < \alpha < 1$ and $M_m(\alpha)$ is the matrix defined in Lemma~\ref{lem:M alpha invertible}. By applying Lemma~\ref{lem:M alpha invertible}, we conclude that $M_m(\alpha)$ is invertible. Thus, $B = (k-\ell_2)M_m(\alpha)$ is also invertible, implying that $\rank B = m$.

To conclude, we have $\rank(A(A^t-\frac{\ell_2}{k}J_{n,m})) = \rank B = m$. Thus, we obtain that $\rank A \ge m$. On the other hand, since $A \in \RR^{m \times n}$, we may deduce that $\rank A \le n$. Combining these inequalities, we obtain $m \le \rank A \le n$, as desired.
\end{proof}

\section{Proof of Theorem~\ref{thm:my pattern theorem}}
\label{sec:proof pattern theorem}

The proof is motivated by and generalizes Bloom and Sagan's proof~\cite[Section 2]{BloomSagan} that, for all $k>3$, every symmetrically avoided set $\Pi \subseteq S_k$ with at most $2$ elements consists only of monotone elements. Recall that the monotone elements are denoted by $\iota_n = 12 \dots n = \mathrm{Id}_n$ and $\delta_n = n \dots 21$.

First, we introduce a combinatorial criterion of symmetry, which is implicitly stated in Gessel---Reutenauer~\cite{gessel1993counting}.
\begin{definition}
\label{def:respect composition}
    Let $\B \subseteq S_n$ be a set of permutations. The subset of permutations that \emph{respect} a given composition $\alpha \vDash n$, denoted $\B(\alpha)$, consists of the permutations $\pi \in \B$ such that $\Des(\pi) \subseteq S_\alpha$, where $S_\alpha$ is the set corresponding to the composition $\alpha$.
\end{definition}

\begin{lemma}
\label{lem:symmetry criterion}
    A set $\B \subseteq S_n$ is symmetric if and only if $|\B(\alpha)| = |\B(\beta)|$ for all $\alpha \sim \beta \vDash n$.
\end{lemma}

\begin{proof}
By Definition~\ref{def:symmetric set}, the set $\B$ is symmetric if and only if its generating function $\mathcal{Q}_n (\B)$ is symmetric. We have:
\begin{align*}
    \mathcal{Q}_n (\B) & =  \sum_{\pi \in \B} F_{\Des(\pi)} & \text{(Definition~\ref{def:generating function})} \\
    & = \sum_{\pi \in \B} \  \sum_{\alpha : S_\alpha \supseteq \Des(\pi)} M_\alpha & \text{(Observation~\ref{obs:decomposition fundamental and monomial})} \\
    & = \sum_{\alpha \vDash n} \big| \left\{\pi \in \B \mid \Des(\pi) \subseteq S_\alpha \right\}\big| \cdot M_\alpha \\
    & = \sum_{\alpha \vDash n} |\B(\alpha)| \cdot M_\alpha.
\end{align*}
The statement now follows from Corollary~\ref{cor:symmetric function criterion}.
\end{proof}

For example, denote $\pi_1 = 1243$ and $\pi_2 = 4213$, and consider the set $\B = \{\pi_1,\pi_2\} \subseteq S_4$. Let $\alpha=(3,1)$ and $\beta = (1,3)$ denote two equivalent compositions. Since $\Des(\pi_1) = \{3\}$, $\Des(\pi_2) = \{1,2\}$ and $S_\alpha = \{3\}$, we have $\B(\alpha) = \{\pi_1\}$. On the other hand, since $S_\beta = \{1\}$, we have $\B(\beta) = \emptyset$. Therefore, $|\B(\alpha)| = 1 \ne |\B(\beta)| = 0$ and the set $\B$ is not symmetric.

To proceed with the proof of Theorem~\ref{thm:my pattern theorem}, we state a few additional lemmas. The first lemma establishes a connection between symmetric sets and intersecting families, as defined in Definition~\ref{def:intersecting family}:
\begin{lemma}
\label{lem:descents are uniform intersecting}
Let $n \in \NN$, and assume that a set $\B = \{\pi_1, \dots, \pi_m\} \subseteq S_n$ is symmetric. For every  $1 \le i \le n-1$, let $A_i := \{1 \le j \le m \mid i \notin \Des(\pi_j) \}$. Then, there exist non-negative integers $k, \ell_1, \ell_2$ such that the family $\{A_i \mid 1 \le i \le n-1\}$ is $k$-uniform and $(\ell_1,\ell_2)$-intersecting.
\end{lemma}

\begin{proof}
According to Definition~\ref{def:intersecting family}, we need to establish three claims: (1) $|A_i|$ is constant for all $i$, (2) $|A_i \cap A_{i+1}|$ is constant for all $i$, and (3) $|A_i \cap A_j|$ is constant for $j \ge i+2$.
We prove each claim separately, employing a similar approach for each.
    \begin{enumerate}
        \item \emph{$|A_i|$ is constant for all $i$:} Consider the compositions of the form $\alpha_i := (1^{i-1}, 2, 1^{n-i-1} )$. These compositions are pairwise equivalent. By applying Lemma~\ref{lem:symmetry criterion}, we conclude that $|\B(\alpha_i)|$ does not depend on $i$. Notably, $\pi \in \B(\alpha_i)$ if and only if $ \pi \in \B$ and $i \notin \Des(\pi)$.
        Thus, we have $\B(\alpha_i) = \{\pi_j \mid j\in A_i\}$, implying that $|\B(\alpha_i)| = |A_i|$ is constant for all $i$.
        \item \emph{$|A_i \cap A_{i+1}|$ is constant for all $i$:} Consider the compositions $\alpha_{i,i+1} := (1^{i-1}, 3, 1^{n-i-2} )$, where $i \le n-2$. These compositions are pairwise equivalent as well. Applying Lemma~\ref{lem:symmetry criterion}, we conclude that $|\B(\alpha_{i,i+1})|$ does not depend on $i$. Notably, $\pi \in \B(\alpha_{i,i+1})$ if and only if $ \pi \in \B$ and $i,i+1 \notin \Des(\pi)$. Thus, we have $\B(\alpha_{i,i+1}) = \{\pi_j \mid j\in A_i \cap A_{i+1}\}$, implying that $|\B(\alpha_{i,i+1})| = |A_i \cap A_{i+1}|$ is constant for all $i$.
        \item \emph{$|A_i \cap A_j|$ is constant for $j \ge i+2$:} Similarly to the previous proofs, consider the pairwise equivalent compositions of the form $\alpha_{i,j} := (1^{i-1}, 2, 1^{j-i-2}, 2, 1^{n-j-1} )$ where $j \ge i+2$. By employing Lemma~\ref{lem:symmetry criterion}, we observe that $|\B(\alpha_{i,j})| = |A_i \cap A_j|$ does not depend on $i$ and $j$.
    \end{enumerate}
\end{proof}

Given a symmetric set $\B$ and utilizing the notation from Lemma~\ref{lem:descents are uniform intersecting}, the family $\{A_i \mid 1 \le i \le n-1\}$ is $k$-uniform and $(\ell_1,\ell_2)$-intersecting, where $k=|A_1|$, $\ell_1 = |A_1 \cap A_2|$ and $\ell_2 = |A_1 \cap A_3|$, as implied by Lemma~\ref{lem:descents are uniform intersecting}. As we proceed to apply Theorem~\ref{thm:my extremal theorem} and provide an upper bound on the number of sets in the family, we first need to ensure that all the sets are distinct:

\begin{lemma}
\label{lem:all columns distinct}
Let $n \ge 5$, let $\B \subseteq S_n$ be a symmetric non-empty set with no monotone elements, and let $i_1 \ne i_2$. Then there exists $\pi \in \B$ such that $i_1 \in \Des(\pi)$ and $i_2\notin \Des(\pi)$.
\end{lemma}
\begin{remark}
The assumption that $n \ge 5$ is necessary, as the statement does not hold for $n=4$. For example, consider the Knuth class $\B = \{3412, 3142\}$, known to be symmetric (and even Schur-positive). We have $\Des(3412) = \{2\}$ and $\Des(3142) = \{1, 3\}$, so every $\pi \in \B$ descends in $1$ if and only if it descends in $3$. Therefore, the statement of the lemma does not apply to $n=4$.
\end{remark}

\begin{proof}
Assume, by contradiction, that $i_1 \in \Des(\pi)$ implies $i_2\in \Des(\pi)$ for all $\pi \in \B$. Using the notation of Lemma~\ref{lem:descents are uniform intersecting}, we have $A_{i_2} \subseteq A_{i_1}$. By Lemma~\ref{lem:descents are uniform intersecting}, there exist $k, \ell_1, \ell_2 \ge 0$, such that the family $\{A_i \mid 1 \le i \le n-1\}$ is $k$-uniform and $(\ell_1,\ell_2)$-intersecting.

If $|i_1-i_2|=1$, then we have $\ell_1 = |A_{i_1} \cap A_{i_2}| = |A_{i_2}| = k$. Otherwise, we have $\ell_2 = k$. In the latter case, we obtain $|A_1 \cap A_4| = \ell_2 = |A_1| = |A_4| = k$ (recall that $n \ge 5$, so the family consists of $n-1 \ge 4$ sets), which implies $A_1 = A_4$. Similarly, $|A_2 \cap A_4| = |A_2| = |A_4|$, leading to $A_2 = A_4$. Thus, we have $A_1 = A_2$ and $\ell_1 = k$.

In either case, we conclude that $\ell_1 = k$. Therefore, for every $1 \le i \le n-2$, we have $|A_i \cap A_{i+1}| = |A_i| = |A_{i+1}|$, which implies $A_i = A_{i+1}$. This implies that all the sets in the family are equal, and every $\pi \in \B$ is monotone, contradicting the assumption of the lemma.
\end{proof}

Let us combine these lemmas to show that, as long as $n \ge 5$, every non-empty symmetric set with no monotone elements has at least $n-1$ elements:

\begin{lemma}
\label{lem:there is no small symmetric set}
For every $n \ge 5$, every non-empty \emph{symmetric} set $\B \subseteq S_n$ with no monotone elements has at least $n-1$ elements.
\end{lemma}
\begin{remark}
It is well known that for $n \ge 5$, the minimal size of a non-empty \emph{Schur-positive} set $\B \subseteq S_n$ with no monotone elements is $n-1$. This minimum is achieved, for instance, by a Knuth equivalence class corresponding to the partition $(n-1, 1)$~\cite{knuth_classes_schur_positive}. The lemma states that the same bound holds for the weaker property of symmetry as well.

The assumption that $n \ge 5$ is necessary here too, as the counter-example for Lemma~\ref{lem:all columns distinct} for $n=4$ holds here as well. Specifically, the symmetric set $\B = \{3412, 3142\}$ has only two elements.
\end{remark}

\begin{proof}[Proof of Lemma~\ref{lem:there is no small symmetric set}]
Let $\B \subseteq S_n$ be a symmetric set with $|\B| = m$. We follow the notation of Lemma~\ref{lem:descents are uniform intersecting} for the sets $A_i \subseteq [m]$. By Lemma~\ref{lem:descents are uniform intersecting}, the family $\{A_i \mid 1 \le i \le n-1\}$ is $k$-uniform and $(\ell_1,\ell_2)$-intersecting. Furthermore, by Lemma~\ref{lem:all columns distinct}, all the sets $A_i$ are distinct. Thus, we can apply Theorem~\ref{thm:my extremal theorem} and deduce that $n-1 \le m$, as required.
\end{proof}

Now we can continue with the proof of the main theorem.

\begin{proof}[Proof of Theorem~\ref{thm:my pattern theorem}]
It suffices to show that the pattern-avoiding set $S_k(\Pi) = S_k \setminus \Pi$ is not symmetric.

Assume by contradiction that $S_k \setminus \Pi$ is symmetric. Therefore, its generating function $\mathcal{Q}_k (S_k \setminus \Pi)$ is symmetric (according to Definition~\ref{def:symmetric set}). Definition~\ref{def:generating function} implies that $\mathcal{Q}_k (S_k \setminus \Pi) = \mathcal{Q}_k (S_k) - \mathcal{Q}_k (\Pi)$. The function $\mathcal{Q}_k(S_k)$ is symmetric (see, e.g., \cite[Section 2]{S3_patterns_list}), so the function $\mathcal{Q}_k (\Pi)$ is symmetric as well. Therefore, the set $\Pi$ is symmetric.

Since the sets $\{\iota_k\}$ and $\{\delta_k\}$ are symmetric, we obtain that the set $\Pi \setminus \{\iota_k, \delta_k\}$ is symmetric as well. In addition, $|\Pi| = p \ge 3$, so $\Pi \setminus \{\iota_k, \delta_k\}$ is a non-empty symmetric set with no monotone elements. By Lemma~\ref{lem:there is no small symmetric set}, we may deduce that $|\Pi| \ge k-1$, which contradicts $|\Pi \setminus \{\iota_k, \delta_k\}| \le p \le k-2$.
\end{proof}

\section{Further remarks and open problems}
\label{sec:conclusion}



\subsection{Intersecting families}
\label{subsec:further research intersceting families}
Theorem~\ref{thm:my extremal theorem} states that any uniform, $(\ell_1, \ell_2)$-intersecting family of distinct subsets of $[n]$ contains at most $n$ sets. This bound is tight, as can be observed from the family $\F = \{\{i\} \mid i \in [n]\}$.

Noga Alon brought to our attention an interesting variant of Theorem~\ref{thm:my extremal theorem}. The variant states that any $(\ell_1, \ell_2)$-intersecting (not necessarily uniform) family of distinct subsets of $[n]$ contains at most $n+2$ sets. The idea for the proof of this variant is that the Gram matrix of the characteristic vectors of the sets has some values on the main diagonal, $\ell_1$ on the two next diagonals and $\ell_2$ everywhere else. If $\ell_1=\ell_2$ then the statement follows from Bose's theorem (Theorem~\ref{thm:his extremal theorem}). Otherwise, subtracting $\ell_2 J$ (where $J$ is the all-ones matrix that has rank $1$), we obtain a similar matrix with $\ell_2$ replaced by $0$ and $\ell_1$ replaced by $\ell_1-\ell_2$. This matrix contains an $(m-1) \times (m-1)$ invertible triangular submatrix that has rank $m-1$, implying that the rank of the Gram matrix is at least $m-2$. Its rank is clearly at most $n$, implying that $m \le n+2$. It is not known whether this bound is tight as well.

After determining the extremal value for a problem in extremal combinatorics, a natural next step is to characterize the objects that achieve this extremum. In the case of Bose's theorem (described in Theorem~\ref{thm:his extremal theorem}), extremal constructions, known as symmetric block designs, have been extensively studied in design theory. Various families of symmetric block designs, such as projective spaces and Hadamard designs, are already known, and several necessary conditions for symmetric block designs have been established (see, e.g.,~\cite[Part II, Section 6]{designs_book} for further details). However, a complete characterization of symmetric block designs remains an open problem.

Given that the bound in Theorem~\ref{thm:my extremal theorem} generalizes Bose's theorem and is tight, it is natural to seek a characterization of its extremal constructions as well. It is worth noting that symmetric block designs naturally arise as extremal $k$-uniform, $(\ell_1,\ell_2)$-intersecting families, satisfying the property $\ell_1 = \ell_2$. Conversely, any extremal family with $\ell_1 = \ell_2$ is a symmetric block design. However, other extremal families with $\ell_1 \ne \ell_2$ may also exist. To the best of our knowledge, there has been no analysis of such families to date. The only such family we have found (modulo permuting the elements of the sets and complementing the sets) is:
\begin{equation}
\label{eqn:exceptional intersecting set}
    \F = \big\{\{1,2\},\; \{3,4\},\; \{1,5\},\; \{2,3\},\; \{1,4\}\big\}.
\end{equation}
Combining the findings of Lemma~\ref{lem:family with wierd intersections is small} with a computer test conducted for the remaining cases, it can be demonstrated that this is the only extremal family with $n>1$ and with the property that $2\ell_2 = \ell_1 + k$, which is suggested as a singular property by the proof of Theorem~\ref{thm:my extremal theorem}. Moreover, this family possesses a compelling algebraic justification for its existence, as it describes (following Lemma~\ref{lem:descents are uniform intersecting}) the descent distribution of any Knuth class of a standard Young tableau (SYT) of shape $(2,2,2)$, known to be Schur-positive~\cite{knuth_classes_schur_positive} (see~\cite{adin_roichman_fine_set} for the relations between SYTs and Schur-positive sets). Except for the descent distributions of Knuth classes of shape $(n-1,1)$ or $(2,1^{n-2})$, every descent distribution of a Schur positive subset of $S_n$ of size $n-1$ with no monotone elements is equivalent to \eqref{eqn:exceptional intersecting set}. These two reasons suggest that \eqref{eqn:exceptional intersecting set} is an exceptional family and raise the following question:
\begin{question}
    Is there a uniform $(\ell_1, \ell_2)$-intersecting family with $\ell_1 \ne \ell_2$ and $n \ne 5$?
\end{question}

A computer test for all $n \le 16$ did not find any such family.

\subsection{Small symmetric sets}
Lemma~\ref{lem:there is no small symmetric set} shows that every \emph{symmetric} set $\B \subseteq S_n$ with no monotone elements is either empty or contains at least $n-1$ elements.
A similar result is known for \emph{Schur-positive} sets. The intriguing connection between the minimal sizes of symmetric and Schur-positive sets raises the following conjecture:
\begin{conjecture}
\label{conj:minimal symmetric set always schur positive}
Every symmetric subset of $S_n$ of size $n - 1$ without monotone elements is Schur-positive.
\end{conjecture}

While Lemma~\ref{lem:there is no small symmetric set} provides a lower bound on the size of non-trivial symmetric sets based on an analysis of intersecting families, applying the same approach to prove Conjecture~\ref{conj:minimal symmetric set always schur positive} is expected to be challenging. Section~\ref{subsec:further research intersceting families} highlights the difficulty in completely characterizing extremal intersecting families. We conducted a computer test to identify all extremal intersecting families (i.e., families corresponding to Theorem~\ref{thm:my extremal theorem}) with at most $14$ sets, and multiple families were discovered. However, none of them corresponded to a symmetric set that is not Schur-positive, thereby confirming Conjecture~\ref{conj:minimal symmetric set always schur positive} for $n < 16$.

In summary, while the conjecture holds for small values of $n$ and is likely to hold for all $n$, proving it in general will likely require a deeper understanding of symmetric sets beyond the analysis of $(\ell_1,\ell_2)$-intersecting families.

Another interesting question regarding sizes of symmetric sets is the following:
\begin{question}
\label{ques:sizes of symmetric sets}
Let $n \in \NN$. For which values $p \in \NN$, does there exist a symmetric set $\B \subseteq S_n$ of size $p$ with no monotone elements?
\end{question}
While Lemma~\ref{lem:there is no small symmetric set} establishes that such sets do not exist for $n \ge 5$ and $0 < p < n-1$, there remain unanswered questions about the possible sizes of symmetric sets with no monotone elements.

We know that symmetric and Schur-positive sets of size $p=n-1$ exist. However, it can be shown that for $n \ge 5$, there are no Schur-positive sets of size $n$ with no monotone elements (this result can be derived, for example, from Adin and Roichman's criterion to Schur-positivity \cite[Theorem 1.5]{adin_roichman_fine_set}). On the other hand, there do exist symmetric sets of size $n$ without monotone elements. One such example is the set given by
\[
\left\{n12\dots(n-1)\right\} \cup \left\{n(n-1)12\dots(n-2), 1n(n-1)23\dots(n-2),\dots, 12\dots(n-2)n(n-1)\right\}.
\]
This set consists of the permutation $n12\dots(n-1)$ and all permutations obtained by inserting the labels $n(n-1)$ as adjacent entries, with $n$ appearing first, in the series $12\dots (n-2)$. This set is symmetric and corresponds to the symmetric function $s_{(n-1,1)} + s_{(n-2,1,1)} - s_{(n-2,2)}$.

However, a general answer to Question~\ref{ques:sizes of symmetric sets} remains unknown. Further investigations are required to determine the possible sizes of symmetric sets without monotone elements.

\subsection{Symmetrically avoided sets}
Theorem~\ref{thm:my pattern theorem} implies that for every $p \ge 3$ and $k \ge p+2$, no $p$-element set $\Pi \subseteq S_k$ is symmetrically avoided. The bound is tight, as there exists a symmetrically avoided $(k-1)$-element set $\Pi \subseteq S_{k}$ for every $k \ge 4$. This result can be attributed to a theorem by Hamaker, Pawlowski, and Sagan:

\begin{theorem}[{\cite[Lemma 5.7]{S3_patterns_list}}]
\label{thm:inverse descent class pattern knuth closed}
For any $J \subseteq [k-1]$, the inverse descent class $D_J^{-1} := \{\pi \in S_k \mid \Des(\pi^{-1}) = J\}$ is symmetrically avoided.
\end{theorem}

\begin{corollary}
\label{cor:minimal avoidance example}
Let $k \in \NN$. The following subset of $S_k$, of size $k-1$, is symmetrically avoided:
\[
D_{\{k-1\}}^{-1} =  \{k12\dots(k-1),1k23\dots (k-1), \dots, 12\dots(k-2)k (k-1) \}.
\]
\end{corollary}

It is worth noting that Hamaker, Pawlowski and Sagan proved a stronger property for inverse descent classes, known as pattern-Knuth closed.
Their result implies Theorem~\ref{thm:inverse descent class pattern knuth closed}, since every pattern-Knuth closed set is symmetrically avoided~\cite [Section 5]{S3_patterns_list}.



For $k \ge 5$, the presented results establish that $D_{\{k-1\}}^{-1}$ is the minimal symmetrically avoided set in $S_k$ that contains at least $3$ elements. Consequently, for $k \ge 5$ and for every set $\Pi \subseteq S_k$ with $3 \le |\Pi| \le k-2$, there exists $n \in \NN$ such that the set $S_n(\Pi)$ is not symmetric. The proofs for this fact and the case of $|\Pi| \le 2$ by Bloom and Sagan \cite{BloomSagan} begin by considering $n=k$ and demonstrating that $S_k(\Pi)$ is not symmetric. However, no information is currently available regarding the symmetry of $S_n(\Pi)$ for $n > k$.

Motivated by this observation, we introduce the concept of being \say{eventually symmetrically avoided}:
\begin{definition}
We say that a set $\Pi \subseteq S_k$ is \emph{eventually symmetrically avoided} if the set $S_n (\Pi)$ is symmetric for every sufficiently large $n$.
\end{definition}

\begin{question}
For which values of $p \in \mathbb{N}$ is there a threshold $K = K(p)$ such that if $|\Pi| = p$ and $\Pi \subseteq S_k$ for $k \geq K$, with $\Pi \nsubseteq \{\iota_k, \delta_k\}$, then $\Pi$ is not eventually symmetrically avoided?
\end{question}

As far as we know, this question remains open, even for $p = 1$. The condition $\Pi \nsubseteq \{\iota_k, \delta_k\}$ is added because every $\Pi \subseteq \{\iota_k, \delta_k\}$ is known to be symmetrically avoided~\cite[Section 2]{S3_patterns_list}.

It is important to note that a pattern set that is eventually symmetrically avoided may not be symmetrically avoided. An example of this is provided by Bloom and Sagan~\cite[Section 6]{BloomSagan}, who constructed a pattern set $\Pi \subseteq S_5$ such that $S_6(\Pi)$ is not symmetric, but $S_n(\Pi)$ is symmetric for all $n \geq 7$. Thus, $\Pi$ is not symmetrically avoided, but it is eventually symmetrically avoided.


Furthermore, our consideration has been limited to the case of $\Pi \subseteq S_k$ for some fixed $k$. However, in the study of pattern avoidance, it is common to explore sets of patterns $\Pi$ that can contain patterns of different sizes, i.e., $\Pi \subseteq \bigcup_{k \in \mathbb{N}} S_k$. 
Hence, there is a desire to extend our lower bound to the case where $\Pi$ belongs to $\bigcup_{k \in \mathbb{N}} S_k$. Such a generalization would be valuable in broadening the scope of our results.

\section{Acknowledgements}
My sincere thanks to my supervisors, Ron Adin and Yuval Roichman, for their constant support. Special thanks to Noga Alon for valuable insights and suggestions. I also wish to thank Yotam Shomroni, Noam Ta-Shma and Danielle West, for enlightening discussions and many editing suggestions.

\printbibliography

@article{BloomSagan,
  title={Revisiting pattern avoidance and quasisymmetric functions},
  author={Bloom, Jonathan S. and Sagan, Bruce E.},
  journal={Annals of Combinatorics},
  volume={24},
  number={2},
  pages={337--361},
  year={2020},
  publisher={Springer}
}

@article{S3_patterns_list,
  %IDS = {hook},
  title={Pattern avoidance and quasisymmetric functions},
  author={Hamaker, Zachary and Pawlowski, Brendan and Sagan, Bruce E.},
  journal={Algebraic Combinatorics},
  volume={3},
  number={2},
  pages={365--388},
  year={2020}
}

@article{extremal_theorem,
  title={On $ t $-designs},
  author={Ray-Chaudhuri, Dijen K. and Wilson, Richard M.},
  journal={Osaka Journal of Mathematics},
  volume={12},
  number={3},
  pages={737--744},
  year={1975},
  publisher={Osaka University and Osaka City University, Departments of Mathematics}
}

@article{alon_proof,
  title={Multilinear polynomials and Frankl-Ray-Chaudhuri-Wilson type intersection theorems},
  author={Alon, Noga and Babai, L{\'a}szl{\'o} and Suzuki, Hiroshi},
  journal={Journal of Combinatorial Theory, Series A},
  volume={58},
  number={2},
  pages={165--180},
  year={1991},
  publisher={Elsevier}
}

@article{erdos_szekeres,
  title={A combinatorial problem in geometry},
  author={Erd\H{o}s, Paul and Szekeres, George},
  journal={Compositio mathematica},
  volume={2},
  pages={463--470},
  year={1935}
}

@article{arc_pattern_avoiding,
  title={Arc permutations},
  author={Elizalde, Sergi and Roichman, Yuval},
  journal={Journal of Algebraic Combinatorics},
  volume={39},
  number={2},
  pages={301--334},
  year={2014},
  publisher={Springer}
}

@inproceedings{problem_find_pattern,
  title={Pattern avoidance and quasisymmetric functions},
  author={Sagan, Bruce E.},
  booktitle={The 13th International Permutation Patterns Conference, London, UK},
  year={2015}
}

@article{adin_roichman_fine_set,
  title={Matrices, characters and descents},
  author={Adin, Ron M. and Roichman, Yuval},
  journal={Linear Algebra and its Applications},
  volume={469},
  pages={381--418},
  year={2015},
  publisher={Elsevier}
}

@book{sagan_book,
  title={The symmetric group: representations, combinatorial algorithms, and symmetric functions},
  author={Sagan, Bruce E.},
  volume={203},
  year={2001},
  publisher={Springer Science \& Business Media}
}

@book{avoid123,
  title={Combinatory Analysis},
  author={MacMahon, Percy A.},
  volume={1},
  year={1916},
  publisher={Cambridge University Press}
}

@book{avoid132,
  title={The Art of Computer Programming},
  author={Knuth, Donald E.},
  volume={1},
  year={1969},
  publisher={Addison-Wesley Publishing Co.}
}

@article{knuth_classes_schur_positive,
  title={Multipartite P-partitions and inner products of skew {S}chur functions},
  author={Gessel, Ira M.},
  journal={Contemp. Math},
  volume={34},
  pages={289-301},
  year={1984}
}

@article{ray_chaudhuri_wilson_single_intersection_size,
  title={A note on {F}isher's inequality for balanced incomplete block designs},
  author={Bose, Raj C.},
  journal={The Annals of Mathematical Statistics},
  volume={20},
  number={4},
  pages={619--620},
  year={1949},
  publisher={Institute of Mathematical Statistics}
}

@book{gershgorin,
  title={Matrix analysis Second edition},
  author={Horn, Roger A. and Johnson, Charles R.},
  year={2013},
  publisher={Cambridge University Press, New York}
}

@book{designs_book,
  title={Handbook of Combinatorial Designs},
  author={Colbourne, Charles and Dinitz, Jeffrey},
  year={2007},
  edition={2},
  publisher={CRC press}
}

@article{gessel1993counting,
  title={Counting permutations with given cycle structure and descent set},
  author={Gessel, Ira M. and Reutenauer, Christophe},
  journal={Journal of Combinatorial Theory, Series A},
  volume={64},
  number={2},
  pages={189--215},
  year={1993},
  publisher={Elsevier}
}
\end{document}